\documentclass[12pt]{article}
\usepackage{amsmath,amssymb,amsthm}
\usepackage[margin=3cm]{geometry}
\usepackage{blindtext}
\usepackage{titlesec}
\title{Sections and Chapters}

\title{\bf {\textup{BSE}}- property of tensor product Banach algebras}
\author{Maryam Aghakoochaki  $^1$, \thanks{2020 Mathematics Subject Classifcation. Primary: 46J05; Secondary: 46J10} ,  \and Ali Rejali $^2$, \thanks{Corresponding author}}

\date{
	$^1$Isfahan University  \\ \texttt{mkoochaki@sci.ui.ac.ir, Orcid: 0000-0002-3851-6550}\\%
	$^2$Department of Pure Mathematics, Faculty of Mathematics and Statistics, University of Isfahan, Isfahan 81746-73441, Iran \\ \texttt{rejali@sci.ui.ac.ir, Orcid: 0000-0001-7270-665X}\\[2ex]% 
\today
}

\newcommand{\tnorm}[1]{{\left\vert\kern-0.25ex\left\vert\kern-0.25ex\left\vert #1 
    \right\vert\kern-0.25ex\right\vert\kern-0.25ex\right\vert}}

\theoremstyle{plain}
\newtheorem{thm}{Theorem}[section]
\newtheorem{pro}[thm]{Proposition}
\newtheorem{cor}[thm]{Corollary}
\newtheorem{lem}[thm]{Lemma}
\theoremstyle{definition}

\newtheorem{ex}[thm]{\bf Example}

\newtheorem{DEF}[thm]{\bf Definition}

\begin{document}
\maketitle

\begin{abstract}
		Let $A$ and $B$ be commutative semisimple Banach algebras. In this paper,   the  $\textup{BSE}$ and $\textup{BED}$- property of  tensor  Banach algebra $A{\otimes}_{\gamma} B$
with respect to the Banach algebras $A$ and $B$ are assesed. In particular, $\textup{BSE}$ and $\textup{BED}$-  structure of vector-valued group algebra $L_{1}(G, A)$ and  $C^{*}$- algebra $C_{0}(X, A)$ characterized.

\end{abstract}

\section{Introduction}
In this paper, $X$ is a metric space with at least two elements and $(A,\|.\|)$ is a commutative semisimple Banach algebra over the scalar field $\mathbb C$. 
Let  $\Delta(A)$ be the character space of $ A$ with the Gelfand topology. $\Delta(A)$ is the set consisting of all non-zero multiplicative
linear functionals on $A$.
   Assume that $C_{b}(\Delta(A))$ is the space consisting of all complex-valued continuous and bounded functions
on $\Delta(A)$, with sup-norm $\|.\|_{\infty}$. 
A continuous linear operator $T$ on $A$ is named a multiplier if for all $x,y\in A$, $T(xy)=xT(y)$; See  \cite{kan}.
The set of all multipliers on $A$ will be expressed as $M(A)$. It is obvious that $M(A)$ is a Banach algebra, and if $A$ is an unital Banach algebra, then 
$M(A)\cong A$.  As observed in   \cite{klar},  for each $T\in M(A)$ there exists a unique bounded  continuous function $\widehat{T}$ on $\Delta(A)$ expressed as:
$$\varphi(Tx)=\widehat{T}(\varphi)\varphi(x),$$
for all $x\in A$ and $\varphi\in \Delta(A)$.
By setting $\{\widehat{T}: T\in M(A)\}$, the $\widehat{M(A)}$ is yield.

      If the Banach algebra $A$ is  semisimple, then the Gelfand map $\Gamma_{A}: A\to\widehat A$, $f\mapsto \hat f$, is injective. Define 
$${\cal {M}}(A) :=\{ \sigma\in C_{b}(\Delta(A)): ~ \sigma.\hat{A}\subseteq \hat A\}.$$
 Then $\widehat{M(A)}= \cal{M}(A)$, whenever $A$ is a semisimple Banach algebra.
A bounded complex-valued continuous function $\sigma$ on $\Delta(A)$  is named a  BSE function, if there exists a positive real number $\beta$ in a sense that for every finite complex-number $c_{1},\cdots,c_{n}$,  and  the same many $\varphi_{1},\cdots,\varphi_{n}$ in $\Delta(A)$ the following inequality
$$\mid\sum_{j=1}^{n}c_{j}\sigma(\varphi_{j})\mid\leq \beta\|\sum_{j=1}^{n}c_{j}\varphi_{j}\|_{A^{*}}$$
holds.\\

The set of all BSE-functions is expressed by $C_{\textup{BSE}}(\Delta( A))$, where for 
 each $\sigma$, the BSE-norm of $\sigma$, $\|\sigma\|_{\textup{BSE}}$ 
is the infimum of all  $\beta$s  applied in the above inequality.   That $(C_{\textup{BSE}}(\Delta(A)), \|.\|_{\textup{BSE}})$ is a semisimple Banach subalgebra of $C_{b}(\Delta(A))$ is in Lemma 1 proved in \cite{E6}. Algebra $A$ is named a BSE algebra  if it meets the following condition:
$$\widehat{M(A)}\cong C_{\textup{BSE}}(\Delta(A)).$$ 
If $A$ is unital, then $\widehat{M(A)}\cong \widehat{A}\mid_{\Delta(A)}$, indicating that $A$ is a  BSE algebra if and only if  $C_{\textup{BSE}}(\Delta(A))\cong \widehat{A}\mid_{\Delta(A)}$.
If there exists some  $K>0$ in a sense that for each $a\in A$, the following holds: 
$$
\|a\|_{A}\leq K\|\hat{a}\|_{\textup{BSE}}
$$
then the semisimple Banach algebra $A$ is named a norm-BSE algebra.

  The  Bochner-Schoenberg-Eberlein (BSE) is derived from the famous theorem proved in 1980 by Bochner and Schoenberg for the group of real numbers; \cite{SB} and \cite{P1}. Later, Eberlein in \cite{N2},  revealed that if $G$ is any locally compact abelian group, then the group algebra $L_{1}(G)$ is a BSE algebra.
 The researchers in \cite{P1},\cite{E7},\cite{E8} assessed the commutative Banach algebras that meet the Bochner-Schoenberg-Eberlein- type theorem and explained their properties.

  In 2007, Inoue and Takahasi in \cite{F1}  introduced and assessed the concept of $\textup{BED}$- concerning Doss, where Fourier-Stieltjes transforms of absolutely continuous measures are specified.
The acronym  $\textup{BED}$  stands for Bochner-Eberlein-Doss and refers to a famous
theorem proved in \cite{RD1} and \cite{RD2}. They proved that if $G$ is a locally compact Abelian
group, then the group algebra $L_{1}(G)$  is a $\textup{BED}$- algebra. Later \cite{AP}, the researcher showed that $l^{p}(X, A)$ is a $\textup{BED}$- algebra if and only if $A$ is so.
The authors in \cite{MAR} and \cite{MAR2}, investigated and assess the correlation between different types of $\textup{BSE}$- Banach algebras $A$, and the Banach
algebras $C_{0}(X,A)$ and $L_{1}(G,A)$.

The basic terminologies and the related information on  $\textup{BED}$- algebras are
extracted from \cite{F1}, \cite{kan}, and \cite{klar}.

The function $\sigma\in C_{BSE}(\Delta(A))$ is  a ${\textup{BED}}$- function, if for all $\epsilon>0$, there exists some compact set such $K\subseteq \Delta(A)$, where  for all $c_{i}\in\mathbb C$ and for all $\varphi_{i}\in\Delta(A)\backslash K$ the following inequality:
$$
|\sum_{i=1}^{n} c_{i}\sigma(\varphi_{i}) |\leq \epsilon\|\sum_{i=1}^{n}c_{i}\varphi_{i}\|_{A^{*}}
$$
holds. This definition of $\textup{BED}$-  functions is a modification of the definition in \cite{RD1}.
  The set of all $\textup{BED}$-  functions is expressed by the $C_{BSE}^{0}(\Delta(A))$.  Clearly, $C_{BSE}^{0}(\Delta(A))$ is a closed ideal of $ C_{BSE}(\Delta(A))$; See[\cite{F1}, Corollary 3.9].

 We will prove that:\\
If $A$ is a commutative semisimple unital dual Banach algebra and $G$ is a locally compact group. Then \\
(i)
$$
C_{\textup{BSE}}(\Delta(L^{1}(G, A))= \widehat{M(G)}\overset{\textup{BSE}}{\otimes}C_{\textup{BSE}}(\Delta(A))
$$
(ii)
$$
C_{\textup{BSE}}^{0}(\Delta(L^{1}(G, A))= \widehat{L^{1}(G)}\overset{\textup{BSE}}{\otimes}C_{\textup{BSE}}^{0}(\Delta(A))
$$
(iii)
$$
C_{\textup{BSE}}(\Delta(C_{0}(X, A)))= \widehat{C_{b}(X)}\overset{\textup{BSE}}{\otimes}C_{\textup{BSE}}(\Delta(A))
$$
(iv)
$$
C_{\textup{BSE}}^{0}(\Delta(C_{0}(X, A)))= \widehat{C_{0}(X)}\overset{\textup{BSE}}{\otimes}C_{\textup{BSE}}(\Delta(A))
$$

In general, for each commutative  semisimple Banach algebras $A$ and $B$, we show that:\\
(i) 
$$
C_{\textup{BSE}}(\Delta(A{\otimes}_{\gamma}B)) = C_{\textup{BSE}}(\Delta(A)) \overset{\textup{BSE}}{\otimes} C_{\textup{BSE}}(\Delta(B)).
$$
(ii)
$$
C_{\textup{BSE}}^{0}(\Delta(A{\otimes}_{\gamma}B)) = C_{\textup{BSE}}^{0}(\Delta(A)) \overset{\textup{BSE}}{\otimes} C_{\textup{BSE}}^{0}(\Delta(B)).
$$
%++++++++++++++++++++++++++++++++++++++++++++++++++++++++++++++++++++++++++++++++++++++
\section{ ${\textup{BSE}}$- properties of Banach algebras}
In this section, some of the basic terminologies of ${\textup{BSE}}$-  algebra
are reviewed and proven. Let $A$ be a commutative semisimple  Banach algebra and $X$ be a locally compact Hausdorff space.
\begin{lem}\label{lsi}
If $\sigma\in C_{\textup{BSE}}(\Delta(A))$, $c_{i}\in\mathbb C$ and $\varphi_{i}\in\Delta(A)$ for all $1\leq i\leq n$, then
$$
\|\sum_{i=1}^{n} c_{i}\sigma(\varphi_{i})\varphi_{i}\|_{A^{*}}\leq\|\sigma\|_{BSE}.\|\sum_{i=1}^{n}c_{i}\varphi_{i}\|_{A^{*}}.
$$
\end{lem}
\begin{proof}
\begin{align*}
\|\sum_{i=1}^{n}c_{i}\sigma(\varphi_{i})\varphi_{i}\|_{A^{*}} &= \textup {sup}\{ \sum_{i=1}^{n}c_{i}\sigma(\varphi_{i})\varphi_{i}(a): ~ \|a\|_{A}\leq 1\}\\
                                                                                                           &= \textup {sup}\{ \sum_{i=1}^{n}c_{i}(\sigma. \hat{a})(\varphi_{i}): ~ \| a\|_{A}\leq 1\}
\end{align*}
Since $\hat{A}\subseteq C_{\textup{BSE}}(\Delta(A))$, so $\sigma.\hat{a}\in C_{\textup{BSE}}(\Delta(A))$. Thus the following is the yield:
\begin{align*}
\|\sum_{i=1}^{n}c_{i}\sigma(\varphi_{i})\varphi_{i}\|_{A^{*}} &\leq \textup {sup}\{\|\sigma.\hat{a}\|_{\textup{BSE}}.\|\sum_{i=1}^{n}c_{i}\varphi_{i}\|_{A^{*}}:~ \|a\|_{A}\leq 1\}\\
                                                                                                       &\leq \textup {sup}\{\|\sigma\|_{\textup{BSE}}.\|\hat{a}\|_{\textup{BSE}}.\|\sum_{i=1}^{n}c_{i}\varphi_{i}\|_{A^{*}}:~ \|a\|_{A}\leq 1\}\\
                                                                                                       &\leq \|\sigma\|_{\textup{BSE}}. \|\sum_{i=1}^{n}c_{i}\varphi_{i}\|_{A^{*}}
\end{align*}
This implies that
$$
\|\sum_{i=1}^{n}c_{i}\sigma(\varphi_{i})\varphi_{i}\|_{A^{*}}\leq  \|\sigma\|_{\textup{BSE}}. \|\sum_{i=1}^{n}c_{i}\varphi_{i}\|_{A^{*}}.
$$
\end{proof}
In the following, we give a simple proof of the result in \cite{F1}, with another definition.
\begin{pro}\label{pcid}
 Let $A$ be a commutative semisimple  Banach algebra. Then $C_{\textup{BSE}}^{0}(\Delta(A))$ is a closed ideal of $ C_{\textup{BSE}}(\Delta(A))$.
\end{pro}
\begin{proof}
Assume that $\sigma\in C_{\textup{BSE}}(\Delta(A))$ and  $\sigma_{0}\in C_{\textup{BSE}}^{0}(\Delta(A))$. So for all $\epsilon>0$ there exists some compact set $K\subseteq \Delta(A)$ where for all $c_{i}\in\mathbb C$
and $\varphi_{i}\in\Delta(A)\backslash K$ the following inequality is holds:
$$
|\sum_{i=1}^{n}c_{i}\sigma_{0}(\varphi_{i})|\leq \epsilon\|\sum_{i=1}^{n}c_{i}\varphi_{i}\|_{A^{*}}
$$
Then
\begin{align*}
|\sum_{i=1}^{n}c_{i}\sigma.\sigma_{0}(\varphi_{i})| &= |\sum_{i=1}^{n}c_{i}\sigma(\varphi_{i})\sigma_{0}(\varphi_{i})|\\
                                                                                        & \leq \epsilon\|\sum_{i=1}^{n}c_{i}\sigma(\varphi_{i})\varphi_{i}\|_{A^{*}}
\end{align*}
Furthermore, according to Lemma \ref{lsi} we have 
$$
|\sum_{i=1}^{n}c_{i}\sigma.\sigma_{0}(\varphi_{i})|\leq  \epsilon\|\sigma\|_{\textup{BSE}}. \|\sum_{i=1}^{n}c_{i}\varphi_{i}\|_{A^{*}}.
$$
Therfore $\sigma\sigma_{0}\in C_{\textup{BSE}}^{0}(\Delta(A))$, so $C_{\textup{BSE}}^{0}(\Delta(A))$ is an  ideal of $ C_{\textup{BSE}}(\Delta(A))$. In the sequel, let $(\sigma_{n})$ be a Cauchy sequence in $C_{\textup{BSE}}^{0}(\Delta(A))$, where $\sigma_{n}\to \sigma$ for some $\sigma\in C_{\textup{BSE}}(\Delta(A))$.
Then  for all $\epsilon>0$ there exists some  number $N$ where for all $n\geq N$ such that 
$$\|\sigma_{n}-\sigma\|_{BSE}\leq\frac{\epsilon}{2}.$$
 Since 
$$\sigma_{N}\in C_{\textup{BSE}}^{0}(\Delta(A))$$
So 
for all $\epsilon>0$ there exists some compact set $K\subseteq \Delta(A)$ where for all $c_{i}\in\mathbb C$
and $\varphi_{i}\in\Delta(A)\backslash K$, the following inequality is holds:
$$
|\sum_{i=1}^{n}c_{i}\sigma_{N}(\varphi_{i})|\leq \frac{\epsilon}{2}\|\sum_{i=1}^{n}c_{i}\varphi_{i}\|_{A^{*}}
$$
This implies that
\begin{align*}
|\sum_{i=1}^{n}c_{i}\sigma(\varphi_{i})| &\leq |\sum_{i=1}^{n}c_{i}\sigma(\varphi_{i})- \sum_{i=1}^{n}c_{i}\sigma_{N}(\varphi_{i})|+|\sum_{i=1}^{n}c_{i}\sigma_{N}(\varphi_{i})|\\
                                                                    &\leq \|\sigma-\sigma_{N}\|_{BSE}.\|\sum_{i=1}^{n}c_{i}\varphi_{i}\|_{A^{*}}+\frac{\epsilon}{2}\|\sum_{i=1}^{n}c_{i}\varphi_{i}\|_{A^{*}}\\
                                                                     &\leq\epsilon\|\sum_{i=1}^{n}c_{i}\varphi_{i}\|_{A^{*}}
\end{align*}
Therefore $\sigma\in C_{\textup{BSE}}^{0}(\Delta(A))$ and so $C_{\textup{BSE}}^{0}(\Delta(A))$ is a closed ideal of $ C_{\textup{BSE}}(\Delta(A))$.
\end{proof}
%^^^^^^^^^^^^^^^^^^^^^^^^^^^^^^^^^^^^^^^^^^^
\begin{thm}\label{tbouc0}
 Let $A$ be a  commutative  Banach algebra. If $A$ has a bounded approximate identity in $A_{0}$, then $\widehat{A}\subseteq C_{\textup{BSE}}^{0}(\Delta(A))$.
\end{thm}
\begin{proof}
Assume that $(e_{r})$ is a bounded approximate identity for $A$ in $A_{0}$ and $a\in A$. Then for all $\epsilon>0$ there exist $r_{0}$ such that $\|a- e_{r_{0}}a\|_{A}<\epsilon$. 
If $P= \sum_{i=1}^{n}c_{i}p_{i}$ where $p_{i}\notin supp (\widehat{e_{r_{0}}})$, then 
\begin{align*}
|\hat {a}(P)| &= |\hat {a}- {\widehat{e_{r_{0}}}}(P){\hat a}(P)|\\
                     &= |(a- e_{r_{0}}a\widehat)(P)|\\
                      &\leq \|P\|.\|(a- e_{r_{0}}a\widehat)\|_{A}\\
                        &\leq \|P\|.\|a- e_{r_{0}}a\|_{A}<\epsilon\|P\|
\end{align*}
Thus ${\hat a}\in C_{\textup{BSE}}^{0}(\Delta(A))$.
\end{proof}
\begin{cor}
Let $G$ be a locally compact group. Then 
$$
C_{0}(G\widehat)\subseteq C_{\textup{BSE}}^{0}(\Delta(C_{0}(G)))
$$
and 
$$
L^{1}(G\widehat)\subseteq C_{\textup{BSE}}^{0}(\Delta(L^{1}(G))).
$$
\end{cor}
\begin{proof}
The algebras $C_{0}(G)$ and $L^{1}(G)$ have the bounded approximate identity in $C_{00}(G)$. Thus $C_{0}(G\widehat)\subseteq C_{\textup{BSE}}^{0}(\Delta(C_{0}(G)))$ and $L^{1}(G\widehat)\subseteq C_{\textup{BSE}}^{0}(\Delta(L^{1}(G)))$.
\end{proof}

%^^^^^^^^^^^^^^^^^^^^^^^^^^^^^^^^^^^^^^^^^^^^^^^^^^^^^^^^
In the sequel, the correlation between the ${\textup{BSE}}$ and ${\textup{BED}}$ is assessed. The following theorem is proved in \cite{F1} and \cite{E7}.
We modify their proof as follows: 
\begin{pro}
Let $A$ be a commutative semisimple ${\textup{BED}}$- algebra. Then\\
(i)  
$$
 C_{\textup{BSE}}(\Delta(A))\subseteq \widehat{M(A)}.
$$ 
(ii) If $A$ has $\Delta$- weak approximate identity, then $A$ is a ${\textup{BSE}}$- algebra.
\end{pro}
\begin{proof}
(i)
Since $A$ is ${\textup{BED}}$- algebra, so $C_{\textup{BSE}}^{0}(\Delta(A))= \hat A$. Then by applying proposition \ref{pcid}, $\hat A$ is an ideal of $ C_{\textup{BSE}}(\Delta(A))$. Thus $ C_{\textup{BSE}}(\Delta(A)). \hat{A}\subseteq \hat A$. Therefore
$$
 C_{\textup{BSE}}(\Delta(A))\subseteq {\cal M}(A)=\widehat{M(A)}.
$$
(ii) By using [\cite{E6}, Corollary 5],  and part (i), it is immediate.
\end{proof}
In the following Lemma, we consider the corelation between $C_{\textup{BSE}}^{0}(\Delta(A))$ and $C_{\textup{BSE}}(\Delta(A))$.

\begin{lem}\label{aden}
 Let $A$ be a commutative semisimple   Banach algebra. Then \\
(i)  If $a\in A$ and $c\in\mathbb C$, then there exists some $\varphi\in\Delta(A)$ where $\varphi(a)=c$.\\
(ii) If $\Delta(A)$ is compact, then 
$$
C_{\textup{BSE}}^{0}(\Delta(A))= C_{\textup{BSE}}(\Delta(A))
$$
(iii)  If $\Delta(A)$ is discrete, then 
$C_{\textup{BSE}}^{0}(\Delta(A))$ is dense in $C_{\textup{BSE}}(\Delta(A))$.\\
(iv) If $A$ is unital, then
$$
C_{\textup{BSE}}^{0}(\Delta(A)) = C_{\textup{BSE}}(\Delta(A))
$$
and  $A$ is a ${\textup{BSE}}$- algebra if and only if $A$ is a ${\textup{BED}}$- algebra.\\
(v) $\hat A$ is dense in $C_{0}(\Delta(A))$.
\end{lem}
\begin{proof}
(i) Let $b= a- ce$ where $e= 1_{A}$ and $I= \{bx: x\in A\}$ be the ideal generated by $b$.  Then there exists a maximal ideal $M$ where $I\subseteq M$. Put $M= ker (p)$ for some $p\in\Delta(A)$, then $p(a)=c$.\\
(ii)  If $\Delta(A)$ is compact clearly,
$$
C_{\textup{BSE}}^{0}(\Delta(A))= C_{\textup{BSE}}(\Delta(A))
$$
(iii) If $\Delta(A)$ is discrete and $p\in \Delta(A)$, then  $\chi_{p}\in C_{\textup{BSE}}^{0}(\Delta(A))$. Thus $C_{00}(\Delta(A))\subseteq C_{\textup{BSE}}^{0}(\Delta(A))$, but $C_{00}(\Delta(A))$ is dense in $C_{0}(\Delta(A))$ and
 $C_{\textup{BSE}}^{0}(\Delta(A))\subseteq C_{0}(\Delta(A))$. Therefore $C_{\textup{BSE}}^{0}$ is dense in  $C_{0}(\Delta(A))$. \\
(v) Since $A$ is  semisimple , so $A= \hat A$. Hence $\hat A$ is a subalgebra of $C_{0}(\Delta(A))$. If $p\neq q$ in $\Delta(A)$, then there exists $a\in A$ wherre $p(a)\neq q(a)$, thus $\hat A$ separates the points of $\Delta(A)$.
 Therefore by Stone- Weirstrass theorem $\hat A$ is dense in $C_{0}(\Delta(A))$.

\end{proof}
%^^^^^^^^^^^^^^^^^^^^^^^^^^^^^^
%^^^^^^^^^^^^^^^^^^^^^^^^^^^^^^^^^^^^^^^^^^^^^^^^
%^^^^^^^^^^^^^^^^^^^^^^^^^^^^^^^^^^^^^^^^^^^^^^^^
The following lemma shows that our definition of $ \textup {BED}$- algebra is equivalent to its definition in \cite{F1}.
\begin{lem}
Let $A$ be a commutative Banach algebra where $\Delta(A)\neq\emptyset$, $K\subseteq \Delta(A)$ be compact set and $\sigma\in C_{\textup{BSE}}(\Delta(A))$ such that\\
(i)
\begin{align*}
 \|\sigma\|_{\textup{BSE}, K} :=  ~ \textup{sup}\{|\sum_{i=1}^{n} c_{i}\sigma(\varphi_{i})| : ~ \|\sum_{i=1}^{n}c_{i}\varphi_{i}\|_{A^{*}}\leq 1, ~ c_{i}\in\mathbb C, ~ \varphi_{i}\in\Delta(A)\backslash K\}
\end{align*}
(ii)
\begin{align*}
 \|\sigma\|_{\textup{BSE}, \infty} := ~ \textup{inf}\{ \|\sigma\|_{\textup{BSE}, K}: ~ K\subseteq \Delta(A) ~ is ~ compact\}.
\end{align*}
Then 
$$
C_{\textup{BSE}}^{0}(\Delta(A))= \{\sigma\in C_{\textup{BSE}}(\Delta(A)): ~ \|\sigma\|_{\textup{BSE}, \infty} =0\}
$$
\end{lem}
\begin{proof}
If $\sigma\in C_{\textup{BSE}}(\Delta(A))$ and $\|\sigma\|_{\textup{BSE}, \infty} =0$, then there exists a sequence of compact sets $(K_{n})$ in $\Delta(A)$ where $\|\sigma\|_{\textup{BSE}, K_{n}}\to 0$. So for all $\epsilon> 0$ there exists some integer number
$N $ where for all $n\geq N$ we have $\|\sigma\|_{\textup{BSE}, K_{n}}\leq\epsilon$.  Set $K:= K_{N}$. At a result
$$
|\sum_{i=1}^{n} c_{i}\sigma(\varphi_{i})|\leq \|\sigma\|_{\textup{BSE}, K}< \epsilon
$$ 
where $\|\sum_{i=1}^{n} c_{i}\varphi_{i}\|_{A^{*}}\leq 1$, $\varphi_{i}\notin K$ and $c_{i}\in \mathbb C$.
Define
$$
F: \overline{<\Delta(A)>}\to\mathbb C
$$
by 
$$
F(\sum_{i=1}^{n} c_{i}\varphi_{i})= \sum_{i=1}^{n} c_{i}\sigma(\varphi_{i}).
$$
Thus $F$ is continuous and so by Hahn- Banach theorem there exists an extension  $F\in A^{**}$ where $\|F\|= \|\sigma\|_{\textup{BSE}}$. If $M:= \|\sum_{i=1}^{n} c_{i}\varphi_{i}\|_{A^{*}}$, then 
$\|\sum_{i=1}^{n}\frac{ c_{i}}{M}\varphi_{i}\|_{A^{*}}\leq 1$ and so 
$$|\sum_{i=1}^{n}c_{i}\sigma(\varphi_{i})|< \epsilon\|\sum_{i=1}^{n}c_{i}\varphi_{i}\|_{A^{*}}$$
So $\sigma\in  C_{\textup{BSE}}^{0}(\Delta(A))$.

     Conversely, if  $\sigma\in  C_{\textup{BSE}}^{0}(\Delta(A))$, then for $\epsilon =\frac{1}{n}$ there exist some compact set \\
$K_{n}\subseteq \Delta(A)$ where for all $c_{i}\in\mathbb C$ and $\varphi_{i}\in \Delta(A)\backslash K_{n}$
the following is the yield:
$$
|\sum_{i=1}^{n} c_{i}\sigma(\varphi_{i})|\leq \frac{1}{n} \|\sum_{i=1}^{n} c_{i}\varphi_{i}\|_{A^{*}}.
$$
This implies that: 
$$
\|\sigma\|_{\textup{BSE}, K_{n}}\leq \frac{1}{n}
$$
Thus 
\begin{align*}
0 &\leq{\textup{inf}}\{ \|\sigma\|_{\textup{BSE}, K}: ~ K\subseteq \Delta(A) ~ is ~ compact\}\\
     &\leq \underset{n}{\textup{inf}}\{ \|\sigma\|_{\textup{BSE}, K_{n}}: ~ K_{n}\subseteq \Delta(A) ~ is ~ compact\}= 0
\end{align*}
Therefore $\|\sigma\|_{\textup{BSE}, \infty}=0$
\end{proof}

%^^^^^^^^^^^^^^^^^^^^^^^^
\begin{DEF}
The Banach algebra $A$ is called ${\textup{BED}}$- norm algebra, if there exist $M> 0$ such that $\|a\|_{A}\leq M\|\hat a\|_{\infty}$ for all $a\in A$.
\end{DEF}
%^^^^^^^^^^^^^^^^^^^^^^^^^^^^^^^^^^^^^
\begin{pro}\label{pbdn}
(i)
Let $A$ be a commutative semisimple  ${\textup{BED}}$- algebra. Then it is a ${\textup{BSE}}$- norm algebra.\\
(ii) Let $A$ be a commutative semisimple  ${\textup{BED}}$- norm algebra. Then 
$A$ is a ${\textup{BED}}$- algebra. 
\end{pro}
\begin{proof}
(i)
By hypothesis,
$\widehat{A}= C_{\textup{BSE}}^{0}(\Delta(A))$ is a Banach algebra, so $(\widehat {A}, \|.\|_{\textup{BSE}})$ is complete. But the map $a\to\hat a$ is a continuous and isomorphism, thus by applying the open mapping theorem it has a continuous inverse map.  
As a result, there exists $M> 0$ such that $\|a\|\leq M\|\hat a\|_{\textup{BSE}}$.\\
(ii)
If $A$ is a ${\textup{BED}}$- norm algebra, then
$$
\|\widehat a\|_{\infty}\leq \|\widehat a\|_{\textup{BSE}}\leq \|a\|_{A}\leq M\|\widehat a\|_{\infty}
$$
Thus $A= (\widehat{A}, \|.\|_{\infty})$ is a sub- $C^{*}$ algebra with equivalent norms. Since  Property $\textup{BSE}$ and property $\textup{BED}$  are stable under uniformity with equivalent norms, $A$ is a  $\textup{BSE}$  and $\textup{BED}$- algebra.
Therefore
$$
A=  C_{\textup{BSE}}^{0}(\Delta(A)).
$$
\end{proof}
%^^^^^^^^^^^^^^^^^^^^^^^^^^^^^^^^^^^^^^^^^^^^^^^^^^^^^^^^^^
\begin{pro}
Let $A$ be a semisimple commutative  Banach algebra. Then \\
(i) $A$ is a ${\textup{BSE}}$- norm algebra if and only if $A\cong (\widehat A, \|.\|_{\textup{BSE}})$.\\
(ii) $A$ is a ${\textup{BED}}$- norm algebra if and only if $A\cong (\widehat A, \|.\|_{\infty})$.
\end{pro}
%^^^^^^^^^^^^^^^^^^^^^^^^^^^^^^^^^^^^^^^^^
%^^^^^^^^^^^^^^^^^^^^^^^^^^^^^^^^^^^^^^^^^^^^^^^^^^^^^^^^^^^^^^^^^^^^^^^^^^^^^^^^^
\begin{DEF}
Assume that $A$ is a commutative Banach algebra with non-empty character space. If  
$ A_{0}:= \{a\in A: ~ Supp(\hat a) ~  is ~ compact\}$ is  norm dense in $A$, then $A$ is  called  Tauberian.
\end{DEF}
Clearly, every unital Banach algebra is Tauberian.
%^^^^^^^^^^^^^^^^^^^^^^^^^^^^^^^^^^^^^^^^^^^^^^^^^^^^^^^^^
\begin{thm}
Let $A$ and $B$ be commutative Banach algebras with non-empty character spaces. Assume that  $\epsilon\leq\gamma\leq\pi$ and $A{\otimes}_{\gamma} B$ is a Banach algebra. Then $A{\otimes}_{\gamma}  B$
is Tauberian if and only if $A$ and $B$ are Tauberian.
\end{thm}
\begin{proof}
Assume that  $A{\otimes}_{\gamma}  B$ is Tauberian, $\varphi\in \Delta(A)$ and $a\in A$ where $\varphi(a)=1$. If $b\in B$ and $\epsilon>0$, then there exist $s\in (A{\otimes}_{\gamma}  B)_{0}$ such that 
$$
\|a\otimes b - s\|_{\gamma}< \epsilon.
$$
Define
\begin{align*}
P: ~ & A{\otimes}_{\gamma} B\to B\\ 
        & x\otimes y\mapsto \varphi(x)y
\end{align*}
Put $t:= P(s)$. As a result, the following is the yield:
\begin{align*}
\|t- b\|_{B} &= \|P(s)- P(a\otimes b)\|_{B}\\
                     & \leq \|P\|. \|s- a\otimes b\|_{\gamma}\\ 
                      &\leq \epsilon \|P\|
\end{align*}
If
\begin{align*}
Q: ~  &\Delta(A) \times \Delta(B)\to \Delta(B)\\
           & \psi_{1}\otimes \psi_{2}\mapsto \psi_{2}
\end{align*}
then
\begin{align*}
\hat{t}(\psi) = \psi(P(s)) = \varphi\otimes \psi(s)= {\hat s}(\varphi\otimes \psi)
\end{align*}
for all $\psi\in\Delta(B)$. As a result
\begin{align*}
supp(\hat t) &= Supp(\widehat{P(s)})\\
                      & \subseteq Q( supp(\hat s))
\end{align*}
Since $Q$ is continuous, $supp(\hat s)$ is compact and $supp(\hat t)$ is closed, so $t\in B_{0}$ and 
$$
\|t- b\|_{B}< \epsilon\|P\|
$$
Thus $B_{0}$ is norm dense in $B$ and therefore $B$ is Tauberian. In a similar way, $A$  is Tauberian.

  Conversely, assume that $A$ and $B$ are Tauberian. \\
(i) If $a\in A_{0}$ and $b\in B_{0}$, then 
$$
supp (a\otimes b\widehat)\subseteq supp(\hat a)\otimes supp(\hat b)
$$
Thus $$a\otimes b\in (A{\otimes}_{\gamma} B)_{0}.$$
(ii)
If $x\in A_{0}\otimes B_{0}$, then 
$$
x=\sum_{k=1}^{n}a_{k}\otimes b_{k}
$$
for some $a_{k}\in A_{0}$ and $b_{k}\in B_{0}$. Then
$$
supp(\hat x)\subseteq \cup_{k=1}^{n} supp(\widehat{a_{k}})\otimes \cup_{k=1}^{n} supp(\widehat{b_{k}})
$$
Therefore 
$$
A_{0}\otimes B_{0}\subseteq (A{\otimes}_{\gamma}  B)_{0}.
$$
(iii) If $x\in A{\otimes}_{\gamma}  B$, then there exist $y\in A\otimes B$ where $\|y-x\|_{\gamma}<\frac{\epsilon}{2}$. 
Assume that $y= \sum_{k=1}^{n} a_{k}\otimes b_{k}$. Thus there exist some sequence $(a_{k}^{l})\in A_{0}$ and $(b_{k}^{l})\in B_{0}$ where
$$ 
a_{k}^{l} \overset{\|.\|_{A}}{\to} a_{k}
$$
and 
$$
b_{k}^{l} \overset{\|.\|_{A}}{\to} b_{k}.
$$
Then 
$$
\sum_{k=1}^{n}a_{k}^{l}\otimes b_{k}^{l}\to \sum_{k=1}^{n}a_{k}\otimes b_{k}.
$$
So there exists $z= \sum_{k=1}^{n}a_{k}^{l_{0}}\otimes b_{k}^{l_{0}}$ where $\|z-y\|_{\gamma}< \frac{\epsilon}{2}$. Thus 
$$z\in A_{0}\otimes B_{0}\subseteq (A\otimes_{\gamma} B)_{0}$$ 
and so $z\in (A{\otimes}_{\gamma}  B)_{0}$. Therefore
$$
\overline{(A{\otimes}_{\gamma} B)_{0}}= A{\otimes}_{\gamma}  B. 
$$
This implies that $A{\otimes}_{\gamma} B$ is Tauberian.
\end{proof}
%^^^^^^^^^^^^^^^^^^^^^^^^^^^^^^^^^^^^^^^^^^^^^^^^^^^^^^^^^^^^^^^^^^^^^^^^^
\begin{lem}
(i)  $C_{0}(X, A)$  is Tauberian if and only if $A$ is so.\\
(ii) $L^{1}(G, A)$  is Tauberian if and only if $A$ is so.\\
\end{lem}
\begin{proof}
By using \cite{JT}, it is immediate.
\end{proof}

%^^^^^^^^^^^^^^^^^^^^^^^^^^^^^^^^^^^^^^^^^^^^^^^^^^^^^^^^^^^^^^^
\begin{lem}\label{lt}
Let $A$ be a commutative  Tauberian algebra. Then\\
(i)
$$
\hat{A}\subseteq  C_{\textup{BSE}}^{0}(\Delta(A)).
$$
(ii)   
$ C_{\textup{BSE}}^{0}(\Delta(A))$ is dense in $C_{0}(\Delta(A).$
\end{lem}
\begin{proof}
(i)
Since  $A_{0}$ is dense in $A$, so $\widehat{A_{0}}$ is dense in $\hat A$. Moreover, $ C_{\textup{BSE}}^{0}(\Delta(A))$ is closed, thus 
$$
\hat A \subseteq  C_{\textup{BSE}}^{0}(\Delta(A)).
$$
(ii) By applying Lemma\ref{aden} and part (i) it is immediate. 
\end{proof}
%^^^^^^^^^^^^^^^^^^^^^^^^^^^^^^^^^^^^^^^^^^^^^^^^^
\begin{cor}
Let $A$ be a commutative semisimple Tauberian Banach algebra. Then\\
(i)
\begin{align*}
C_{0}(X, A\widehat{)}\subseteq C_{BSE}^{0}(\Delta C_{0}(X, A)).
\end{align*}
(ii)
\begin{align*}
{L^{1}(G, A\widehat)}\subseteq  C_{BSE}^{0}(\Delta(L^{1}(G, A))
\end{align*}
\end{cor}
In the following example, the distinction between ${\textup{BSE}}$ and ${\textup{BED}}$- algebras is specified; see \cite{E7} and \cite{F1}.
%^^^^^^^^^^^^^^^^^^^^^^^^^^^^^^^^^^^
\begin{ex}
(i) If $G$ is a non-discrete Abelian group, then $A= (M(G), \star)$ is a  commutative semisimple unital Banach algebra such that $A$ is not a ${\textup{BSE}}$- algebra and  ${\textup{BED}}$-  algebra.\\
(ii) If $A= L^{1}(G)\cap L^{p}(G)$ where $1< p<\infty$ where $G$ is an  infinite compact Abelian group, then $A$ with convolution product and the following norm
$$
\|f\|= max\{\|f\|_{1}, \|f\|_{p}\}
$$
 is a commutative semisimple  Banach algebra, such that $A$  is a ${\textup{BSE}}$- algebra, but is not a ${\textup{BED}}$- algebra.\\
(iii) If $S$ is infinite set, then $A= (l_{1}(S), \|.\|_{1})$ with pointwise multiplication is a Banach algebra, such that  $A$ is not a ${\textup{BSE}}$- algebra but it is a ${\textup{BED}}$-  algebra.\\
(iv) If $X$ is a locally compact Hausdorff space and $A= C_{0}(X)$, then $A$ is a ${\textup{BSE}}$- algebra and ${\textup{BED}}$-  algebra.
\end{ex}
%+++++++++++++++++++++++++++++++++++++++++++++++++++++++++++++++++++++++++++++++++++++++++++

%+++++++++++++++++++++++++++++++++++++++++++++++++++++++++++++++++++++++++++++++++++++++++++++++

\section{\textup{BSE}- properties  of  tensor product Banach algebra}
Let  $A$ and $B$ be commutative Banach algebra with nonempty character spaces. Then for each cross-norm $\|.\|_{\gamma}$ for which $\epsilon\leq \gamma\leq \pi$,
where $\epsilon$ [resp. $\pi$] is injective [resp. projective] cross-norms on $A\otimes B= <x\otimes y: x\in A, y\in B>$, where $A\otimes B$ is a subalgebra of bilinear Banach algebra $BL(A^{*}, B^{*}, \mathbb C)$
defined by, 
$$
x\otimes y: A^{*}\times B^{*}\to \mathbb C
$$
for which $x\otimes y(f,g):= f(x)g(y)$ for $f\in A^{*}$ and $g\in B^{*}$.
Then $A\otimes_{\gamma}B$ is the completion of $A\otimes B$ under  the cross- norm $\|.\|_{\gamma}$. Clearly, $A\otimes B$ is an algebra under the product 
$$
(\sum_{i=1}^{n}a_{i}\otimes b_{i}).(\sum_{j=1}^{m}c_{j}\otimes d_{j}):= \sum_{i= 1}^{n} \sum_{j=1}^{m}a_{i}c_{j}\otimes b_{i}d_{j}
$$
In general, $A\otimes_{\gamma}B$ need not a Banach algebra; see \cite{VR}. But $A\otimes_{\pi}B:= A\hat \otimes B$ is a Banach algebra and each $u\in A\otimes_{\pi}B$ can represented as $u=\sum_{n=1}^{\infty}a_{n}\otimes b_{n}$
such that $\sum_{n=1}^{\infty}\|a_{n}\|_{A}\|b_{n}\|_{B}< \infty$.  Clearly, $\hat{u}= \sum_{n=1}^{\infty}\hat{a_{n}}\otimes \hat{b_{n}}$ and $\|\hat u\|_{\infty}\leq \| u\|_{\pi}$.
In fact, 
$$
(A\otimes_{\pi} B\hat) = \hat A \otimes_{\pi} \hat B
$$  
In particular, 
$$
(a\otimes b \hat)= \hat a\otimes \hat b.
$$
In general,
if $x\in A\otimes_{\gamma} B$, then $x_{n}\to x$ for some sequence $(x_{n})$ in $A\otimes B$. But 
$$
\|\widehat{x_{n}}-\hat x\|_{\infty}< \|x_{n}- x\|_{\gamma}\to 0
$$
Thus $\hat x= lim \widehat{x_{n}}$. Therefore $(A\otimes_{\gamma} B\widehat) =\hat A\otimes_{\gamma}\hat B$.
Let $A\otimes_{\gamma}B$ be a Banach algebra. Then $\Delta(A\otimes_{\gamma}B)= \Delta(A)\times \Delta(B)$; See\cite{kan}. Let $\varphi\in\Delta(A)$ and $\psi\in\Delta(B)$. Then
$$
\varphi\odot\psi: A\times B\to \mathbb C
$$
defined by
$$
\varphi\odot\psi(a, b):= \varphi(a)\psi(b)
$$
for $a\in A$ and $b\in B$.  
  Also, $C_{0}(X, A)^{*}= M(X, A^{*})$ as Banach spaces under the duality,
$$
U: M(X, A^{*})\to C_{0}(X, A)^{*}
$$
for which $\mu\mapsto U_{\mu}$ and for $g\in C_{0}(X)$ and $a\in A$
$$
U_{\mu}(g\otimes a):= \int_{X}\int_{A^{*}}g(x)f(a)d\mu(x,f)
$$
Also, $L^{1}(G, A)= L^{1}(G)\hat\otimes A$ and  $C_{0}(X, A)= C_{0}(X)\check\otimes A$ are Banach algebra; see\cite{kan}. Furthermore, $L^{1}(G, A)^{*}= L^{\infty}(G, A^{*})$, where $A$ is a
 separable Banach algebra.

Under the  duality 
\begin{align*}
T &: L_{\infty}(G, A^{*})\to L^{1}(G, A)^{*}\\
     & g\mapsto T_{g}
\end{align*}
for all $f\in L^{1}(G, A)$ and $g\in  L_{\infty}(G, A^{*})$. There exist sequence $(g_{n})$ in $L^{\infty}(G)\otimes A^{*}$ such that 
$$
\|g_{n}- g\|_{\epsilon}\to 0
$$
Let $f:= \sum_{m=1}^{\infty}f_{m}\otimes a_{m}$ where $\sum_{m=1}^{\infty}\|f_{m}\|.\|a_{m}\|<\infty$. Then 
\begin{align*}
T_{g}(f) &:= \underset{n}{lim} T_{g_{n}}(f)\\
               &=  \underset{n}{lim}\sum_{m=1}^{\infty}g_{n}(f_{m}\otimes a_{m})
\end{align*}
where 
\begin{align*}
g_{n}(f_{m}\otimes a_{m}) &:= \sum_{k=1}^{l_{n}}h_{k}^{n}\otimes F_{k}^{n}(f_{m}\otimes a_{m})\\
                                               &= \sum_{k=1}^{l_{n}}h_{k}^{n}(f_{m}) F_{k}^{n}(a_{m})
\end{align*}
for  some $h_{k}^{n}\in L^{\infty}(G)$ and $F_{k}^{n}\in A^{*}$.
%^^^^^^^^^^^^^^^^^^^^^^^^^^^^^^^^^^^^^^^^^^^^^
\begin{lem}
Let $A$ and $B$ be commutative Banach algebras. Then\\
(i) $A\otimes_{\pi} B$ is semisimple if the following two conditions are satisfied.\\
(1) $A$ and $B$ are semisimple.\\
(2) The natural homomorphism  $A\otimes_{\pi}B\to A\otimes B$ is injective.\\
(ii) $A\otimes_{\pi} B$ is unital if and only if both $A$ and $B$ are unital.\\
(iii) $ A\otimes_{\pi} B$ has bounded approximate identity if and only if both $A$ and $B$ have so.
\end{lem}
\begin{proof}
(i) In Theorem 2.11.6 in \cite{kan} was proved.\\
(ii) It was shown that in Theorem  2.11. 7 in \cite{kan}.\\
(iii) By applying Theorem 2. 8. 10 in \cite{kan}, it is proved.
\end{proof}
%______________________________________________________
\begin{DEF}
Let $A$ and $B$ be commutative Banach algebra with nonempty character spaces. Then for all $F\in A^{**}$ and $G\in B^{**}$\\
(i) 
$$F\odot G(\sum_{i=1}^{n}\sum_{j=1}^{m} c_{i,j}F(\varphi_{i}\otimes\psi_{j})) := \sum_{i=1}^{n}\sum_{j=1}^{m} c_{i,j}F(\varphi_{i})G(\psi_{i})$$
where $\varphi_{i}\in \Delta(A)$ and $\psi_{j}\in\Delta(B)$.\\
(ii)
$$
\|\sum_{i=1}^{k}F_{i}\odot G_{i}\|_{\textup{BSE}}:= sup\{|\sum_{i=1}^{k}F_{i}\odot G_{i}(p)|:~ p\in <\Delta(A)\times\Delta(B)>_{1}\}.
$$
(iii)  We denote $A^{**}\overset{\textup{BSE}}{\otimes}B^{**}$ as the complete metric space of $(<A^{**}\odot B^{**}>, \|.\|_{\textup{BSE}})$. In a special case, since $C_{\textup{BSE}}(\Delta(A))= A^{**}|_{\Delta(A)}\cap C_{b}(\Delta(A)$
and $C_{\textup{BSE}}(\Delta(B))= B^{**}|_{\Delta(B)}\cap C_{b}(\Delta(B)$. We define
$$
C_{\textup{BSE}}(\Delta(A))\overset{\textup{BSE}}{\otimes}C_{\textup{BSE}}(\Delta(B)):= \|.\|_{\textup{BSE}}- cl\{<C_{\textup{BSE}}(\Delta(A)){\odot}C_{\textup{BSE}}(\Delta(B))>\}
$$
and
$$
A^{**}\overset{\textup{BSE}}{\otimes}B^{**}:= \|.\|_{\textup{BSE}}- cl\{<A^{**}\odot B^{**}>\}.
$$
$\beta X$ is the Ston- Cech compactification of $X$ and 
$$
\beta X:= w^{*}- cl\{\delta_{x}: x\in X\}\subseteq {C_{b}(X)}^{*}
$$
where 
$\delta_{x}(f)= f(x)$, for $f\in C_{b}(X)$.
\end{DEF}
%^^^^^^^^^^^^^^^^^^^^^^^^^^^^^^^^^^^
\begin{lem}
Let $X$ be a completely regular Hausdorff space and $Y$ be a compact Hausdorff space. Then\\
(i)
$$
\beta (X\times Y)= \beta X\times Y
$$
and so
$$
C_{b}(X\times Y)= C_{b}(X)\check\otimes C(Y).
$$  
(ii)
$$
C_{b}(X\times Y)= C(\beta X\times Y).
$$
isometrically isomorphism as Banach algebras.
\end{lem}
\begin{proof}
(i)
 Since $X\subseteq \beta X$, so $X\times Y \subseteq \beta X\times Y$. Thus $\beta(X\times Y)\subseteq \beta X\times Y$.

   Conversely, if $z\in \beta X$, and $\widehat{x_{r}}\to z$, then $(x_{r}, y\widehat)\to (z, y\widehat)$. In fact, for each $f\in C_{b}(X\times Y)$ we have
$$
(x_{r}, y\widehat)(f)= f(x_{r}, y)= \widehat{x_{r}}(fy)\to z(fy)= (z, y\widehat)(f).
$$
Then $\beta X\times Y\subseteq \beta(X\times Y)$. Therefore
\begin{align*}
C_{b}(X\times Y) &= C(\beta X\times Y)\\
                             &= C(\beta X)\check\otimes C(Y)= C_{b}(X)\check\otimes C(Y).
\end{align*}
(ii)
Define 
\begin{align*}
F: C_{b}(X\times Y) &\to  C(bX\times Y)\\
                                 & f\mapsto f^{'}
\end{align*}
where $f^{'}(z, y)= lim f(x_{r},y)$ such that  $\widehat{x_{r}}\overset{w^{*}}{\to} z$ in $C_{b}(X)^{*}$. Thus $f^{'}(z,y)= z(f_{y})$ where $f_{y}(x)= f(x,y)$. Clearly,  $f_{y}$ is continuous and
$$
\|f^{'}\|_{\infty}\leq \|z\|.\|f_{y}\|\leq \|f\|_{\infty}
$$
Otherwise $\|f\|\geq|f(x_{r},y)|$, so $\|f^{'}\|_{\infty}\geq \|f\|_{\infty}$ and $F$ is isometry. Clearly, $F$  is continuous.

\end{proof}
%^^^^^^^^^^^^^^^^^^^^^^^^^^^^^^^^^^^^^^^^^^^^^
%^^^^^^^^^^^^^^^^^^^^^^^^^^^^^^^^^^^^^^^
\begin{lem}
$$
<\Delta(A\otimes_{\gamma}B)>= <\Delta(A)>\odot<\Delta(B)>
$$
\end{lem}
\begin{proof}
If $V:= <\Delta(A)>\odot<\Delta(B)>$, then $V$ is a vector space. Thus $\Delta(A)\times \Delta(B)\subseteq V$, and so
$<\Delta(A\otimes_{\gamma}B)>\subseteq V$.

  Conversely, $p\odot q= p\otimes q$ on $A\times B$, so $V\subseteq <\Delta(A\otimes_{\gamma}B)>$. This completes the
proof.
\end{proof}
Furthermore, $\varphi\odot\psi= \varphi\otimes \psi|_{A\otimes B}$, where $\varphi\otimes\psi: A^{**}\times B^{**}\to\mathbb C$ defined by $\varphi\otimes \psi(F, G)= F(\varphi)\psi(G)$, for all $F\in A^{**}$
and $G\in B^{**}$. Thus
$$
\varphi\odot\psi(a\otimes b)= \varphi\otimes\psi(\hat a\otimes\hat b)
$$
for all $a\in A$ and $b\in B$.
  Clearly, $\Delta(A)\subseteq C_{\textup{BSE}}(\Delta(A))^{*}$, by embeding, $\varphi\mapsto\hat\varphi$ where $\hat{\varphi}(\sigma)= \sigma(\varphi)$ for all $\sigma\in  C_{\textup{BSE}}(\Delta(A))$.
Let $\sigma_{1}\in  C_{\textup{BSE}}(\Delta(A))$ and $\sigma_{2}\in  C_{\textup{BSE}}(\Delta(B))$. Then 
$$
\sigma_{1}\otimes \sigma_{2}:  C_{\textup{BSE}}(\Delta(A))^{*}\times  C_{\textup{BSE}}(\Delta(B))^{*}\to\mathbb C
$$
defined by $\sigma_{1}\otimes \sigma_{2}(F, G)= F(\sigma_{1})G(\sigma_{2})$. We denote  $\sigma_{1}\odot \sigma_{2}:= \sigma_{1}\otimes \sigma_{2}|_{\Delta(A)\times \Delta(B)}$.
The linear space generated by $\sigma_{1}\odot \sigma_{2}$ is denoted by $C_{\textup{BSE}}(\Delta(A))\odot C_{\textup{BSE}}(\Delta(B))$.

   If there exists a sequence $(\sigma_{n})$ in $C_{\textup{BSE}}(\Delta(A))\odot C_{\textup{BSE}}(\Delta(B))$ such that $\sigma_{n}\to \sigma$. Then 
$$
\sigma\in C_{\textup{BSE}}(\Delta(A))\overset{{\textup{BSE}}}{\odot} C_{\textup{BSE}}(\Delta(B))
$$
If $p\in <\Delta(A)>$ and $q\in <\Delta(B)>$, then
$$
\|(p,q)\|= \|p\|_{A^{*}}+\|q\|_{B^{*}}
$$
Thus,
$\sigma\in C_{\textup{BSE}}(\Delta(A\hat\otimes B))$ if and only if for all $\epsilon>0$ there exists $\delta>0$ where if
$$
\|\sum_{i=1}^{n}c_{i}\varphi_{i}\|_{A^{*}}+ \|\sum_{i=1}^{m}d_{i}\psi_{i}\|_{B^{*}}< \delta
$$
then
$$
|\sum_{i=1}^{n}\sum_{j=1}^{m}c_{i}d_{j}\sigma(\varphi_{i}\otimes \psi_{j})|<\epsilon
$$
%^^^^^^^^^^^^^^^^^^^^^^^^^^^^^^^^
\begin{thm}
Let $A$ and $B$ be commutative Banach algebra with nonempty character spaces. Then
$$
 C_{\textup{BSE}}(\Delta(A))\overset{{\textup{BSE}}}{\otimes}C_{\textup{BSE}}(\Delta(B))=  C_{\textup{BSE}}(\Delta(A\otimes_{\gamma}B)).
$$
\end{thm}
\begin{proof}
If $\sigma_{1}\in C_{\textup{BSE}}(\Delta(A))$
and $\sigma_{2}\in C_{\textup{BSE}}(\Delta(B))$, then 
\begin{align*}
\|\sigma_{1}\otimes\sigma_{2}\|_{\textup{BSE}}\leq \|\sigma_{1}\|_{\textup{BSE}} \|\sigma_{2}\|_{\textup{BSE}}.
\end{align*}
Thus 
$$
 C_{\textup{BSE}}(\Delta(A)){\otimes}C_{\textup{BSE}}(\Delta(B))\subseteq C_{\textup{BSE}}(\Delta(A\otimes_{\gamma}B)).
$$
As a result
$$
 C_{\textup{BSE}}(\Delta(A))\overset{{\textup{BSE}}}{\otimes}C_{\textup{BSE}}(\Delta(B))\subseteq  C_{\textup{BSE}}(\Delta(A\otimes_{\gamma}B)).
$$
Conversely, if $\sigma\in C_{\textup{BSE}}(\Delta(A\otimes_{\gamma}B))$, then there exists some bounded net $f_{s}$ in $A\otimes_{\gamma}B$ such that ${\widehat{f_{s}}}(\varphi\otimes\psi)\to \sigma(\varphi\otimes\psi)$, for each 
$\varphi\in\Delta(A)$ and each $\psi\in \Delta(B)$. If 
$$f=\sum_{i=1}^{n}a_{i}\otimes b_{i}$$
 then 
$${\hat f}= \sum_{i=1}^{n}{\hat a_{i}}\otimes{\hat b_{i}}\in {\widehat A}\otimes {\widehat B}$$
 where
$$ {\widehat A}\otimes {\widehat B}\subseteq  C_{\textup{BSE}}(\Delta(A)){\otimes}C_{\textup{BSE}}(\Delta(B)).$$
 Thus 
$$\sigma\in T_{p}- cl {\widehat A}\otimes {\widehat B}\subseteq C_{\textup{BSE}}(\Delta(A))\check{\otimes}C_{\textup{BSE}}(\Delta(B))$$
If $\|\sigma_{n}-\sigma\|_{\epsilon}\to 0$ in $C_{\textup{BSE}}(\Delta(A)){\otimes}C_{\textup{BSE}}(\Delta(B))$, Set 
$$S_{n}= \sigma_{n}|_{\Delta(A)\times\Delta(B)}.$$
 Thus
$$
\|S_{n}- \sigma\|_{\textup{BSE}}< \|\sigma_{n}-\sigma\|_{\epsilon}\to 0
$$
Then $\sigma\in  C_{\textup{BSE}}(\Delta(A))\overset{{\textup{BSE}}}{\otimes}C_{\textup{BSE}}(\Delta(B))$. This completes the proof.
\end{proof}

%^^^^^^^^^^^^^^^^^^^^^^^^^^^
\begin{cor}
Let $G$ be a  locally compact Abelian group and $A$  be commutative unital Banach algebra.  Then:
$$
 C_{\textup{BSE}}(\Delta(L^{1}(G, A))) = \widehat{M(G)}\overset{{\textup{BSE}}}{\odot}C_{\textup{BSE}}(\Delta(A)) .
$$
\end{cor}
%$$$$$$$$$$$$$$$$$$$$$$$$$$$$$$$$$$$$$$$$$$$$$$$$$$$$$$$$$$$$$$$$$$$$$$$$$$$$$$$$$$$$$$$$$$$$$$$$$
\begin{cor}
Let $X$  be a locally compact Hausdorff topological space and $A$ be a commutative unital Banach algebra. Then 
$$
C_{\textup{BSE}}(\Delta(C_{0}(X, A)) = C_{b}(X\widehat)\overset{{\textup{BSE}}}{\odot}C_{\textup{BSE}}(\Delta(A)) .
$$
\end{cor}
%^^^^^^^^^^^^^^^^^^^^^^^^^^^^^^^^^^^^^^^^^
\begin{lem}
Let $A$ and $B$ be commutative Banach algebra and $A\otimes_{\gamma}B$ be a  Banach algebra, where $\epsilon\leq\gamma\leq\pi$.
Then  \\
(i)
\begin{align*}
C_{\textup{BSE}}( \Delta(A\check\otimes B)) &\subseteq C_{ \textup{BSE}}( \Delta(A\otimes_{\gamma} B))\\
                                                                           &\subseteq C_{\textup{BSE}}( \Delta(A\hat\otimes B)).
\end{align*}
(ii)
\begin{align*}
C_{\textup{BSE}}^{0}( \Delta(A\check\otimes B)) &\subseteq C_{ \textup{BSE}}^{0}( \Delta(A\otimes_{\gamma} B)) \subseteq C_{\textup{BSE}}^{0}( \Delta(A\hat\otimes B)).
\end{align*}
\end{lem}
\begin{proof}
(i)
If $\sigma\in C_{\textup{BSE}}( \Delta(A\check\otimes B))$,since 
$$
\Delta(A\check\otimes B)= \Delta(A\otimes_{\gamma}B)= \Delta(A)\times \Delta(B)
$$
so
\begin{align*}
|\sum_{i=1}^{n}c_{i}\sigma(\varphi_{i}\otimes\psi_{i})| \leq M\|\sum_{i=1}^{n}c_{i}(\varphi_{i}\otimes\psi_{i})\|_{\epsilon}\leq  M\|\sum_{i=1}^{n}c_{i}(\varphi_{i}\otimes\psi_{i})\|_{\gamma}
\end{align*}
Thus $\sigma\in C_{\textup{BSE}}( \Delta(A\otimes_{\gamma} B))$. In a similar way, because $\|.\|_{\gamma}<\|.\|_{\pi}$, so 
$$
C_{\textup{BSE}}( \Delta(A\otimes_{\gamma} B))\subseteq C_{\textup{BSE}}( \Delta(A\hat\otimes B)).
$$
(ii)
The proof is similar to part (i).
\end{proof}

%^^^^^^^^^^^^^^^^^^^^^^^^
\begin{lem}\label{lc02}
Let $A$ and $B$ be commutative Banach algebra with nonempty character spaces  such that $A$ or $B$ is unital and $A\otimes_{\gamma}B$ be a  Banach algebra, where $\epsilon\leq\gamma\leq\pi$.
Then 
$$
C_{\textup{BSE}}^{0}(\Delta(A))\overset{{\textup{BSE}}}{\odot} C_{\textup{BSE}}^{0}(\Delta(B))\subseteq  C_{\textup{BSE}}^{0}(\Delta(A\otimes_{\gamma}B))
$$
\end{lem}
\begin{proof}
Assume that $B$ is unital, $\sigma_{1}\in C_{\textup{BSE}}^{0}(\Delta(A))$ and $\sigma_{2}\in C_{\textup{BSE}}^{0}(\Delta(B))$. Then
$$\sigma_{1}\odot\sigma_{2}(\varphi\otimes\psi)= \sigma_{1}(\varphi)\sigma_{2}(\psi)$$
 for all $\varphi\in\Delta(A)$ and $\psi\in\Delta(B)$. Since $\sigma_{1}\in C_{\textup{BSE}}^{0}(\Delta(A))$,
so for all $\epsilon>0 $ there exist some compact set $K_{1}\subseteq \Delta(A)$ such that for all $c_{i}\in\mathbb C$ and $\varphi_{i}\in\Delta(A)\backslash K_{1}$ the following is yield:
$$
|\sum_{i=1}^{n}c_{i}\sigma_{1}(\varphi_{i})|<\epsilon\|\sum_{i=1}^{n}c_{i}\varphi_{i}\|_{A^{*}}
$$
Since $B$ is unital, then $\Delta(B)$ is compact. Set $K:= K_{1}\times \Delta(B)$, thus $K\subseteq \Delta(A)\times \Delta(B)$ is compact such that for all $c_{i}\in\mathbb C$ and $\varphi_{i}\otimes\psi_{i}\notin K$ we have
\begin{align*}
|\sum_{i=1}^{n}c_{i}\sigma_{1}\odot\sigma_{2}(\varphi_{i}\otimes\psi_{i})| &= |\sum_{i=1}^{n}c_{i}\sigma_{1}(\varphi_{i})\sigma_{2}(\psi_{i})|\\
                                                                                                                              &\leq \epsilon \|\sum_{i=1}^{n}c_{i} \sigma_{2}(\psi_{i})\varphi_{i}\|_{A^{*}}\\    
                                                                                                                              &\leq \epsilon \|\sigma_{2}\|_{\textup{BSE}} \|\sum_{i=1}^{n}c_{i}\varphi_{i}\otimes\psi_{i}\|_{(A\otimes_{\gamma}B)^{*}}
\end{align*}
Also
$$
\|\sigma_{1}\odot\sigma_{2}\|_{\textup{BSE}}\leq \|\sigma_{1}\|_{\textup{BSE}}\|\sigma_{2}\|_{\textup{BSE}}
$$ 
So $\sigma_{1}\odot\sigma_{2}\in C_{\textup{BSE}}(\Delta(A\otimes_{\gamma}B))$. Therefore $\sigma_{1}\odot\sigma_{2}\in C_{\textup{BSE}}^{0}(\Delta(A\otimes_{\gamma}B))$. As a result,
\begin{align*}
C_{\textup{BSE}}^{0}(\Delta(A))\overset{{\textup{BSE}}}{\odot} C_{\textup{BSE}}^{0}(\Delta(B)) \subseteq C_{\textup{BSE}}^{0}(\Delta(A\otimes_{\gamma}B))
\end{align*}
In a similar way, If $A$ is unital, the above inclusion is established.
\end{proof}
%^^^^^^^^^^^^^^^^^^^^^^^^^^^^^^^^^^^^^^^^^^^^^^^^
\begin{lem}\label{lc01}
Let $A$ and $B$ be commutative Banach algebra with nonempty character spaces  and $A\otimes_{\gamma}B$ be a  Banach algebra, where $\epsilon\leq\gamma\leq\pi$.
Then 
$$
C_{\textup{BSE}}^{0}(\Delta(A\otimes_{\gamma}B))\subseteq C_{\textup{BSE}}^{0}(\Delta(A))\overset{{\textup{BSE}}}{\odot} C_{\textup{BSE}}^{0}(\Delta(B))
$$
\end{lem}
\begin{proof}
Assume that $\sigma\in C_{\textup{BSE}}^{0}(\Delta(A\otimes_{\gamma}B)) $. Thus 
\begin{align*}
\sigma\in C_{0}(\Delta(A\otimes_{\gamma}B) &= C_{0}(\Delta(A)\times \Delta(B))\\
                                                                           &= C_{0}(\Delta(A))\check\otimes C_{0}(\Delta(B))\\
                                                                           &= C_{0}(\Delta(A)\widehat)\check\otimes C_{0}(\Delta(B)\widehat)\\
                                                                           &= C_{\textup{BSE}}^{0}(\Delta(C_{0}(\Delta(A)) \check\otimes C_{\textup{BSE}}^{0}(\Delta(C_{0}(\Delta(B))\\
                                                                           &= C_{\textup{BSE}}^{0}(\Delta(A)) \check\otimes C_{\textup{BSE}}^{0}(\Delta(B))
\end{align*}
Then there exists some sequence $(\sigma_{n})$ in $C_{\textup{BSE}}^{0}(\Delta(A)) \otimes C_{\textup{BSE}}^{0}(\Delta(B))$, where $\sigma_{n}\overset{\|.\|_{\epsilon}}{\to}\sigma$.
 Assume that $\sigma_{n}= \sum_{k=1}^{l_{n}}\sigma_{k}^{n}\otimes \eta_{k}^{n}$, where $\sigma_{k}^{n}\in  C_{\textup{BSE}}^{0}(\Delta(A)) $ and $\eta_{k}^{n}\in  C_{\textup{BSE}}^{0}(\Delta(B))$.
If $ S_{n}= \sum_{k=1}^{p_{n}} \mu_{k}^{n}\odot \nu_{k}^{n}= \sigma_{n}|_{\Delta(A)\times\Delta(B)}$, then
\begin{align*}
 \sum_{k=1}^{l_{n}}\sigma_{k}^{n}\otimes \eta_{k}^{n}(\widehat{\varphi}, \widehat{\psi})  &:=  \sum_{k=1}^{l_{n}}\widehat{\varphi}(\sigma_{k}^{n}) \widehat{\psi}(\eta_{k}^{n})\\
                                                                                                                                                          &:= \sum_{k=1}^{l_{n}} \sigma_{k}^{n}(\varphi)\eta_{k}^{n}(\psi)\\
                                                                                                                                                          &:= \sum_{k=1}^{l_{n}}\mu_{k}^{n}(\varphi)\nu_{k}^{n}(\psi)\\
                                                                                                                                                          &= S_{n}(\varphi, \psi)
\end{align*}
 It is clear that ${\widehat\varphi}\in {C_{\textup{BSE}}^{0}(\Delta(A)) }^{*}$ and 
${\widehat\psi}\in {C_{\textup{BSE}}^{0}(\Delta(B))}^{*}$. Therefore
$$
 S_{n}= \sum_{k=1}^{l_{n}} \mu_{k}^{n}\odot \nu_{k}^{n}\in <C_{\textup{BSE}}^{0}(\Delta(A)) \odot C_{\textup{BSE}}^{0}(\Delta(B))>
$$
where $S_{n}(\varphi\otimes\psi)= \sigma_{n}({\widehat\varphi},{\widehat\psi})$, for all $\varphi\in \Delta(A)$, $\psi\in\Delta(B)$ and $n\in\mathbb N$. The following is yield:
\begin{align*}
\sigma:= \underset{n}{lim}\sigma_{n}= \underset{n}{lim}(\sum_{k=1}^{l_{n}}\sigma_{k}^{n}\otimes \eta_{k}^{n})
\end{align*}
At a result
\begin{align*}
\sigma(\varphi\otimes\psi) = \sigma({\widehat\varphi},{\widehat\psi}) &= \underset{n}{lim}\sigma_{n}({\widehat\varphi},{\widehat\psi})\\
                                                                    &= \underset{n}{lim}  \sum_{k=1}^{l_{n}} \mu_{k}^{n}\odot \nu_{k}^{n}(\varphi\otimes\psi)=  \underset{n}{lim} S_{n}(\varphi,\psi)
\end{align*}

and 
$$
\|S_{n}- \sigma\|_{\textup{BSE}}\leq\|\sigma_{n}- \sigma\|_{\epsilon}\to 0
$$
At a result 
$$
\sigma\in  C_{\textup{BSE}}^{0}(\Delta(A))\overset{{\textup{BSE}}}{\odot} C_{\textup{BSE}}^{0}(\Delta(B))
$$
Therefore 
$$
C_{\textup{BSE}}^{0}(\Delta(A\otimes_{\gamma}B))\subseteq C_{\textup{BSE}}^{0}(\Delta(A))\overset{{\textup{BSE}}}{\odot} C_{\textup{BSE}}^{0}(\Delta(B)).
$$
\end{proof}

%^^^^^^^^^^^^^^^^^^^^^^^^^^^^^^^^
\begin{thm}
Let $A$ and $B$ be commutative Banach algebra with nonempty character spaces such that $A$ or $B$ is unital and $A\otimes_{\gamma}B$ be a  Banach algebra, where $\epsilon\leq\gamma\leq\pi$.
Then 
$$
C_{\textup{BSE}}^{0}(\Delta(A))\overset{{\textup{BSE}}}{\odot} C_{\textup{BSE}}^{0}(\Delta(B)) =  C_{\textup{BSE}}^{0}(\Delta(A\otimes_{\gamma}B))
$$
\end{thm}
\begin{proof}
By applying Lemmas \ref{lc01} and \ref{lc02}, complete the proof.
\end{proof}
%^^^^^^^^^^^^^^^^^^^^^^^^^^^^^^
\begin{cor}\label{ccb0}
Let $X$ be a locally compact Hausdorff topological space and $A$ be a commutative unital Banach algebra. Then 
$$
C_{\textup{BSE}}^{0}(\Delta(C_{0}(X, A))= C_{0}(X\widehat)\overset{{\textup{BSE}}}{\odot}  C_{\textup{BSE}}^{0}(\Delta(A)).
$$ 
\end{cor}
%^^^^^^^^^^^^^^^^^^^^^^^^^^^^^^^^^^^^^^^^^^^6
\begin{cor}
Let $G$ be a  locally compact Abelian group and $A$ be a commutative Banach unital algebra. Then 
$$
C_{\textup{BSE}}^{0}(\Delta(L^{1}(G, A)=  L^{1}(G\widehat) \overset{\textup{BSE}}{\otimes} C_{\textup{BSE}}^{0}(\Delta(A)).
$$
\end{cor}

%$$$$$$$$$$$$$$$$$$$$$$$$$$$$$$$$$$$$$$$$$$$$$$$$$$$$$$$$$$$$$$$$$$$$$
%^^^^^^^^^^^^^^^^^^^^^^^^^^^^^^^^^^^^^^^^^

%^^^^^^^^^^^^^^^^^^^^^^^^^^^^^^^^^^^^^^^^^^^^^^^^^^^^^^^^^^^^^^^^^^^^^^^^^^^^^^^^^^^^^^^^^

\end{document}